\documentclass{amsart}
\RequirePackage[colorlinks,citecolor=blue,urlcolor=blue]{hyperref}
\usepackage[utf8x]{inputenc}
\usepackage{amsmath,amsthm}
\usepackage{amsfonts,amstext,amssymb, mathtools}
\usepackage{url}

\newcommand{\levy}{L\'{e}vy }
\newcommand{\p}{{\mathbb P}}
\newcommand{\e}{{\mathbb E}}
\newcommand{\D}{{\mathrm d}}

\newcommand{\1}[1]{\mbox{\rm\large  1}_{\{#1\}}}

\renewcommand{\l}{\lambda}
\renewcommand{\t}{\theta}
\renewcommand{\a}{\alpha}
\renewcommand{\b}{\beta}
\newcommand{\eqd}{\stackrel{d}{=}}

\newcommand{\T}{T}
\newcommand{\tauR}{\widehat \tau}
\newcommand{\phiR}{{\widehat\phi}}
\newcommand{\eRefR}{\widehat \e^a}
\newcommand{\eRef}{\e^a}
\newcommand{\pRefR}{\widehat \p^a}
\newcommand{\pRef}{\p^a}

\newtheorem{cor}{Corollary}
\newtheorem{lemma}{Lemma}
\newtheorem{prop}{Proposition}
\newtheorem{remark}{Remark}

\begin{document}
\title[Exit problems under continuous and Poisson observations]{Strikingly simple identities relating exit problems for L\'evy processes under continuous and Poisson observations}

%\runauthor{H.\ Albrecher and J.\ Ivanovs}

\author[H. Albrecher]{Hansj\"{o}rg Albrecher}
\email{hansjoerg.albrecher@unil.ch}
\address{University of Lausanne and Swiss Finance Institute}
\author[J. Ivanovs]{Jevgenijs Ivanovs}
\email{jevgenijs.ivanovs@unil.ch}
\address{University of Lausanne}

%\thankstext{m1}{}

%\affiliation{Department of Actuarial Science, University of Lausanne, Lausanne CH-1015, Switzerland\thanksmark{m1}, and\\ Swiss Finance Institute, University of Lausanne, Lausanne CH-1015, Switzerland\thanksmark{m2}}

%\affiliation{University of Lausanne\thanksmark{m1} and Swiss Finance Institute\thanksmark{m2}}

\begin{abstract}
We consider exit problems for general \levy processes, where the first passage over a threshold is detected either immediately or at an epoch of an independent homogeneous Poisson process. It is shown that the two corresponding one-sided problems are related through a surprisingly simple identity. Moreover, we identify a simple link between two-sided exit problems with one continuous and one Poisson exit. Finally, Poisson exit of a reflected process is connected to the continuous exit of a process reflected at Poisson epochs, and a link between some Parisian type exit problems is established. With the appropriate perspective, the proofs of all these relations turn out to be quite elementary. For spectrally one-sided \levy processes this approach enables alternative proofs for a number of previously established identities, providing additional insight.
\end{abstract}

%\begin{keywords}[class=MSC]
%\kwd[Primary ]{}%levy processes
%\kwd[; secondary ]{}%risk theory insurance
%\end{keyword}


\thanks{Financial support by the Swiss National Science Foundation Project 200020 143889 is gratefully acknowledged.}

\maketitle

\section{Introduction}
Let $X=(X_t,t\geq 0)$ be a real-valued \levy process, and let $\T_i,i\geq 1$ be the epochs of an independent Poisson process with intensity~$\lambda>0$; add $\T_0=0$. The probability law corresponding to $X$ started at $u$ will be denoted by $\p_u$ (with $\e_u$ denoting the expectation).
When $u$ is not mentioned explicitly we assume that $u=0$ and write simply $\p $ and $\e$.
 Define 
\begin{align*}&\tau_0^-=\inf\{t\geq 0:X_t<0\}, &\tau_a^+&=\inf\{t\geq 0:X_t>a\},\\
&\tauR_0^-=\min\{\T_i,i\in\mathbb N_0:X_{T_i}<0\}, &\tauR_a^+&=\min\{\T_i,i\in\mathbb N_0:X_{T_i}>a\},
\end{align*}
which we interpret as the first passage times under continuous and Poisson observations, respectively.
Observe that $\tau_0^-<\tauR_0^-$ and, moreover, $\tauR_0^-$ converges in probability to $\tau_0^-$ as $\lambda\rightarrow\infty$ (the same is true for $\tau_a^+$ and $\tauR_a^+$). Thus exit theory under Poisson observation can be regarded as a generalization of the classical exit theory.
Throughout this paper, however, we keep $\lambda>0$ fixed.

Observation at Poisson epochs is both of theoretical and practical interest.
Firstly, some exit problems with Poisson observation yield transforms
of certain occupation times, e.g.
\[\p_u(\tau_0^-<\tauR_a^+)=\e_u\left[\exp\left(-\lambda\int_0^{\tau_0^-}\1{X_t>a}\D t\right);\tau_0^-<\infty\right].\]
Secondly, Poisson observation is relevant in various applications such as queueing (see e.g. \cite{bekker09}), reliability and insurance risk theory (see e.g. \cite{albrecher2011randomized,saj13}). 
In particular, 
in many applications discrete-time observation of stochastic processes would often be considered more natural, but for equidistant discrete time epochs the explicit and tractable analytical structure of continuous-time processes is typically destroyed, so that one is forced towards numerical techniques for the determination of exit probabilities and related quantities. The Poisson observation structure is a bridge between continuous-time and discrete-time observation that still leads to rather explicit, and as will be shown below, also somewhat elegant modifications of the continuous-time formulas.

% In this paper we focus on interpretations in the context of the latter. For example, consider a situation where the surplus process of an insurance company is monitored by an external regulator arriving at rate~$\lambda$, which motivates one to look at, e.g.,~$\p_u(\tauR_0^-=\infty)$.

\subsection{Overview and organization}
In order to stress the intuition behind the derivation of the identities, we will start with a simple case and gradually generalize the setup. Most of the results are stated in terms of relations between transforms, but can also be understood as relations between the corresponding laws in an obvious way. \\
Some of the wording throughout the manuscript will be in terms of the insurance application, where $X$ is the surplus process of a portfolio of  insurance contracts, $\tau_0^-$ is the time of ruin of the portfolio, $\{\tau_0^-=\infty\}$ is the event of (infinite-time) survival, and $\tauR_0^-$ is the time of observed ruin under Poisson observation of the surplus process (in the application the Poisson epochs can for instance be interpreted as the observation times of the regulatory authority). \\

In Section~\ref{sec:exit1} we discuss survival probabilities corresponding to the two observation types, and then proceed to the general one-sided exit problems including the time of exit and the overshoot.
In Section~\ref{sec:exit2} we consider more complex problems. Firstly, the two-sided exit problem with one continously observed and one \mbox{(Poisson-)}discretely observed boundary is related to the one where the observation types at the boundaries are interchanged. Secondly, we provide a link between Poisson exit of a reflected process and continuous exit of the process reflected at Poisson epochs. We also show that a two-sided problem with Poisson exit at both boundaries yields an identity as well, but with a non-standard first passage time. The latter quantity is then linked to a Parisian ruin problem with Erlang-distributed implementation delay. 
Finally, we establish a link between Parisian ruin problems with continuous and Poisson observations. We conclude with Section~\ref{sec:spectral1}, where we specialize to the case of spectrally-one sided processes and demonstrate the use of our simple identities, providing simpler proofs and additional insight to some identities established in earlier literature. 
% That is, we quickly re-derive one-sided exit formulas under Poisson observation using the classical ones under continuous observation.

\subsection{Preliminaries}
The Wiener-Hopf factorization plays a crucial role in the derivations below. % It is, however, essential to put it into an appropriate perspective for our context. 
Define
\begin{align*}&\underline X_t=\inf\{X_s,s\in[0,t]\}, &\underline G_t=\inf\{s\in[0,t]:X_s\wedge X_{s-}=\underline X_t\},
\end{align*}
 the infimum and its (first) time of occurrence up to horizon~$t$. Similarly, the supremum and its (last) time of occurrence are defined by  
 \begin{align*}
&\overline X_t=\sup\{X_s,s\in[0,t]\}, &\overline G_t=\sup\{s\in[0,t]:X_s\vee X_{s-}=\overline X_t\}.
\end{align*}
Finally, let the pairs $(D,T^D)$ and $(U,T^U)$ be distributed as $(\underline X_{\T_1},\underline G_{\T_1})$ and $(\overline X_{\T_1},\overline G_{T_1})$ respectively (under $\p$), and sampled independently of each other and of everything else ($D$ and $U$ stand for `down' and `up').
Recall that according to the Wiener-Hopf factorization we have
\begin{align*}(\underline X_{\T_1},\underline G_{\T_1},X_{\T_1}-\underline X_{\T_1},\T_1-\underline G_{\T_1})&\eqd(D,T^D,U,T^U)\\&\eqd (X_{\T_1}-\overline X_{\T_1},\T_1-\overline G_{\T_1},\overline X_{\T_1},\overline G_{\T_1}),\end{align*}
see, e.g.,~\cite[Thm.\ VI.5]{bertoin}, and \cite{dieker} for applications of factorization embeddings.

\section{One-sided exit}\label{sec:exit1}
\subsection{Survival probability}
Let us first consider
\begin{align}&\phi(u):=\p_u(\tau_0^-=\infty), &\phiR(u):=\p_u(\tauR_0^-=\infty),\label{surv}\end{align}
which in the insurance application are the probabilities of survival with initial capital $u$ under continuous and Poisson observation, respectively.
In fact, the two quantities are connected by two very simple relations:
\begin{prop}\label{prop:survival}
For $u\geq 0$ it holds that 
\begin{align}
&\phiR(u)=\e\phi(u+U),\label{rel1}\\ &\phi(u)=\e\phiR(u+D).\label{rel2}\end{align}
\end{prop}

\begin{figure}[hb]
\centering
\includegraphics{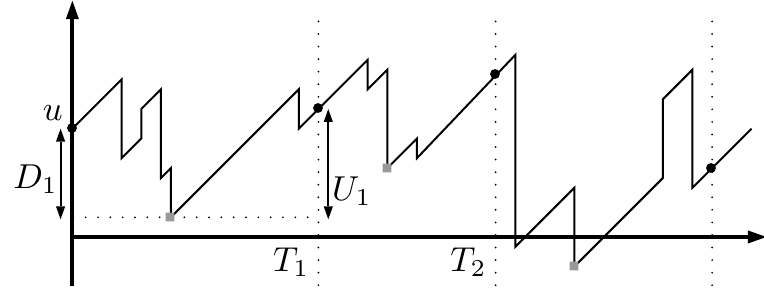}
\caption{Schematic sample path and embeddings}
\label{fig:WH_points}
\end{figure}
\begin{proof}
Survival under Poisson observation is determined by the sequence $u+X_{\T_i}$, whereas survival under continuous observation is determined by the sequence of infima in between the observation epochs (black and grey dots in Figure~\ref{fig:WH_points}, respectively). Let $D_1=\underline X_{T_1}, U_1=X_{T_1}-\underline X_{T_1}$ and 
define $D_{i+1},U_{i+1}$ $(i\ge 1)$ in the same way but for the shifted process $X_{T_i+t}-X_{T_i}$ and exponential time $T_{i+1}-T_i$.
Let $(\widehat S_i,i\in\mathbb N_0)$ and $(S_i,i\in\mathbb N_0)$ be the partial sum processes corresponding to
\[(0,D_1+U_1,D_2+U_2,D_3+U_3,\ldots)\quad \text{and}\quad (D_1,U_1+D_2,U_2+D_3,\ldots),\]
respectively; $u+\widehat S_i$ and $u+S_i$ are the heights of the black and grey dots in Figure~\ref{fig:WH_points}.
Observe that all $U_i$ and $D_i$ are independent, because of independence of increments and  the Wiener-Hopf factorization.
Since the $D_i$'s have the law of $D$ we obtain
\begin{equation}\label{eq:S}(\widehat S_i+D,i\geq 0)\eqd(S_i,i\geq 0).\end{equation}
Similarly, 
\begin{equation}\label{eq:hatS}(\widehat S_i, i\geq 1)\eqd(S_i+U, i\geq 0).\end{equation}
%but the index for $\widehat S_i$ starts from 1. 
Hence
\[\phi(u)=\p(u+\min_{i\geq 0}S_i\geq 0)=\p(u+D+\min_{i\geq 0}\widehat S_i\geq 0)=\e\phiR(u+D),\]
and, since $u\geq 0$, 
\begin{align*}\phiR(u)&=\p(u+\min_{i\geq 0} \widehat S_i\geq 0)=\p(u+\min_{i\geq 1}\widehat S_i\geq 0)\\&=\p(u+U+\min_{i\geq 0}S_i\geq 0)=\e\phi(u+U).\end{align*}
\end{proof}

\begin{remark}\normalfont
Relation \eqref{rel1} allows to interpret the transition from continuous to discrete Poisson observation simply as a (random) increase of the starting value (initial capital) $u$ by $U$, as far as the survival probability is concerned; that is the structure of $\phi$ as a function of $u$ is otherwise completely preserved. Likewise, Relation \eqref{rel2} shows that moving from discrete Poisson to continuous observation preserves the structure, reducing the initial capital by $D$ (which has all its probability mass on the negative half-line). 
\end{remark}

\begin{remark}\normalfont
Suppose we modify the Poisson observation model, so that there is no observation at time~$0$. Then \eqref{rel1} is still valid (even for negative $u$ then), whereas \eqref{rel2} does not hold any more.
\end{remark}

\begin{remark}\normalfont
By the same token one can connect the finite-time survival probabilities $\phi(u,T_i):=\p_u(\tau_0^->T_i)=\p(u+\min_{j\leq i-1}S_j\geq 0)$ and $\phiR(u,T_i):=\p_u(\tauR_0^->T_i)=\p(u+\min_{j\leq i}\widehat S_j\geq 0)$ for $i\in\mathbb N$:
\[\phiR(u,T_i)=\e \phi(u+U,T_i),\quad\phi(u,T_i)=\e\phiR(u+D,T_{i-1}).\]
That is, survival under continuous observation up to an independent Erlang distributed time horizon is intimately related to survival under Poisson observation up to a certain arrival epoch. 
\end{remark}

\subsection{The general identities}
\begin{prop}\label{prop:exit1}
For $u\geq 0$ and $\a,\b\geq 0$ it holds that
\begin{align}
\label{eq:exit1a}&\e_u\left(e^{-\a \tauR_0^-+\b X_{\tauR_0^-}};\tauR_0^-<\infty\right)\\&\qquad =\e\left[e^{-\a T^U}\e_{u+U}\left(e^{-\a \tau_0^-+\b X_{\tau_0^-}};\tau_0^-<\infty\right)\right]\e e^{-\a T^D+\b D},\nonumber\\
\label{eq:exit1b}&\e_{u}\left(e^{-\a \tau_0^-+\b X_{\tau_0^-}};\tau_0^-<\infty\right)\e e^{-\a T^D+\b D}\\&\qquad =\e\left[ e^{-\a T^D}\e_{u+D}\left(e^{-\a \tauR_0^-+\b X_{\tauR_0^-}};\tauR_0^-<\infty\right)\right].\nonumber
\end{align}
\end{prop}
Observe that the left-hand side of~\eqref{eq:exit1b} gives the transform of the undershoot of the first grey point below 0 in Figure~\ref{fig:WH_points}, according to the strong Markov property applied at $\tau_0^-$. Now one can establish the relation between black and grey points as in the proof of Proposition~\ref{prop:survival}, additionally taking time into account.
Essentially, we just shift the picture so that we start at the first grey point.
\begin{proof}[Proof of Proposition~\ref{prop:exit1}]
Let $T^D_1=\underline G_{T_1}, T^U_1=T_1-\underline G_{T_1}$ and define $T^D_{i+1},T^U_{i+1}$ in the same way but for the shifted process $X_{T_i+t}-X_{T_i}$ and exponential time $T_{i+1}-T_i$.
As in the proof of~Proposition~\ref{prop:survival} we consider the sequences of black and grey dots in Figure~\ref{fig:WH_points}, but now we also add the time component: $(\widehat S_i,T_i),i\in\mathbb N_0$ and $(S_i,G_i),i\in\mathbb N_0$ which are the partial sum processes corresponding to
\begin{align*}
&((0,0),(D_1+U_1,T^D_1+T^U_1),(D_2+U_2,T^D_2+T^U_2),\ldots)\quad\text{and}\\
&((D_1,T^D_1),(U_1+D_2,T^U_1+T^D_2),\ldots),
\end{align*}
respectively. Similarly to~\eqref{eq:S} and~\eqref{eq:hatS} we observe that
\begin{align*}
((\widehat S_i+D,T_i+T^D),i\geq 0)&\eqd((S_i,G_i),i\geq 0)\\
((\widehat S_i,T_i),i\geq 1)&\eqd((S_i+U,G_i+T^U),i\geq 0).
\end{align*}
Letting $\widehat N_u=\min\{i\geq 0:u+\widehat S_i<0\}=\min\{i\geq 1:u+\widehat S_i<0\}$ and $N_u=\min\{i\geq 0:u+S_i<0\}$ be the first passage epochs we can write
\begin{align*}
&\e_u\left(e^{-\a \tauR_0^-+\b X_{\tauR_0^-}};\tauR_0^-<\infty\right)=
\e\left(e^{-\a T_{\widehat N_u}+\b (u+\widehat S_{\widehat N_u})};\widehat N_u<\infty\right)\\
&=\e\left(e^{-\a(T^U+G_{N_{u+U}})+\b (u+U+S_{N_{u+U}})};N_{u+U}<\infty\right)\\
&=\e\e_{u+U}\left(e^{-\a (T^U+\tau_0^-)+\b X_{\tau_0^-}};\tau_0^-<\infty\right)\e e^{-\a T^D+\b D},
\end{align*}
where in the last line we applied the strong Markov property at $\tau_0^-$. Identity \eqref{eq:exit1b} can be derived analogously.  
\end{proof}

\section{Further exit problems}%Two-sided exit and first passage for reflected processes}
\label{sec:exit2}
% with different observation regimes}
\subsection{Two-sided exit with different observation types}\label{sec:two_sided}
In this section we consider two-sided exit problems with one continuous and one Poisson exit at the boundaries. It turns out that there is a simple relation between the problems when the roles of the continuous and the Poisson exit are interchanged, i.e. problems corresponding to $\{\tau_0^-<\tauR_a^+\}$ and $\{\tauR_0^-<\tau_a^+\}$. Here we extend the ideas of Section~\ref{sec:exit1} to their full potential.
\begin{prop}\label{prop:exit2}
For $a\geq u\geq 0$ and $\a,\b\geq 0$ it holds that
\begin{align}
\label{eq:exit2a}&\e_u\left(e^{-\a \tauR_0^-+\b X_{\tauR_0^-}};\tauR_0^-<\tau_a^+\right)\\
&\qquad =\e\left[e^{-\a T^U}\e_{u+U}\left(e^{-\a \tau_0^-+\b X_{\tau_0^-}};\tau_0^-<\tauR_a^+\right)\right]\e e^{-\a T^D+\b D},\nonumber\\
\label{eq:exit2b}&\e_{u}\left(e^{-\a \tau_0^-+\b X_{\tau_0^-}};\tau_0^-<\tauR_a^+\right)\e e^{-\a T^D+\b D}\\&\qquad =\e\left[e^{-\a T^D}\e_{u+D}\left(e^{-\a \tauR_0^-+\b X_{\tauR_0^-}};\tauR_0^-<\tau_a^+\right)\right].\nonumber
\end{align}
\end{prop}
\begin{proof} The proof is by inspection: For \eqref{eq:exit2a}, consider the embeddings illustrated in Figure~\ref{fig:exit2}. In the left picture the grey dots correspond to the observations and the black to the suprema in between two observations. In the right picture the black dots correspond to observations and the grey to the infima in between observations.
Note that the position of a black point with respect to the previous grey point has the same distribution in both cases, namely $(U,T^U)$. The same is true for the position of the grey points with respect to their previous black points with $(D,T^D)$. So the patterns of points in each case have the same law up to a certain shifting; we illustrate this by using the same patterns of points in both pictures in Figure~\ref{fig:exit2} and by drawing different sample paths.
Now it follows that $\p_u(\tauR_0^-<\tau^+_a)$, see the left picture, must coincide with $\e\p_{u+U}(\tau_0^-<\tauR_a^+)$, see the right picture, because the interpretation of points was 'reversed'. 
Finally, we include the value of $X$ at first passage and its time using the strong Markov property at $\tau_0^-$ as in the proof of Proposition~\ref{prop:exit1}.
The same type of reasoning yields~\eqref{eq:exit2b}.
\end{proof}
\begin{figure}[h]
\centering
\includegraphics{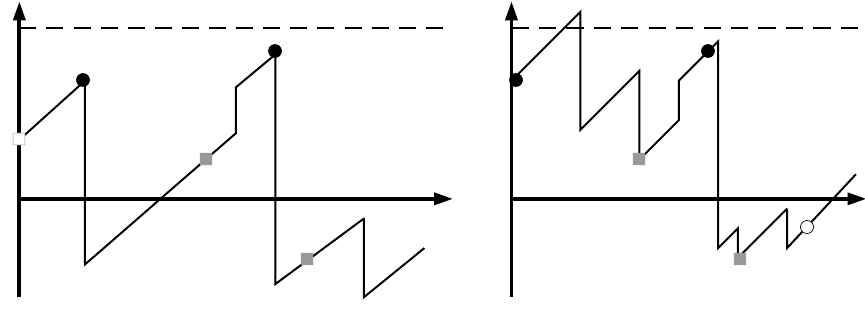}
\caption{Poisson observations are in grey (left picture) and in black (right picture)}
\label{fig:exit2}
\end{figure}

By considering the negative of $X$ we immediately obtain the following result from Proposition~\ref{prop:exit2}.
\begin{cor}
\label{cor:exit2}
For $a\geq u\geq 0$ and $\a,\b\geq 0$ it holds that
\begin{align}
&\e_u\left(e^{-\a \tauR_a^+-\b (X_{\tauR_a^+}-a)};\tauR_a^+<\tau_0^-\right)\\&\qquad =\e\left[e^{-\a T^D}\e_{u+D}\left(e^{-\a \tau_a^+-\b (X_{\tau_a^+}-a)};\tau_a^+<\tauR_0^-\right)\right]\e e^{-\a T^U-\b U},\nonumber\\
&\e_{u}\left(e^{-\a \tau_a^+-\b (X_{\tau_a^+}-a)};\tau_a^+<\tauR_0^-\right)\e e^{-\a T^U-\b U}\\
&\qquad =\e\left[e^{-\a T^U}\e_{u+U}\left(e^{-\a \tauR_a^+-\b (X_{\tauR_a^+}-a)};\tauR_a^+<\tau_0^-\right)\right].\nonumber
\end{align}
\end{cor}

\subsection{Reflected processes}
\label{sec:refl}
In this section we consider the process $X$ reflected at a barrier $a>0$ in a continuous and Poisson manner, and study its first passage below 0 in Poisson and continuous manner respectively (with always opposite manners). 
Again, these two problems are closely related.
Note that in an insurance context reflection at $a$ results when paying out dividends according to a {\em barrier strategy}, either continuously or at Poisson epochs (see e.g.~\cite{albrecher2011randomized}).

Let $\pRef_u$ be the law of $X$ started in $u$ and continuously reflected at $a$ and let $R$ be the corresponding regulator, i.e. $(X_t,R_t),t\geq 0$ under $\pRef_u$ is
\begin{align*}(X_t-(\overline X_t-a)^+,(\overline X_t-a)^+),t\geq 0 \text{ under }\p_u.\end{align*} Similarly, let $\pRefR_u$ be the law of $X$ started at $u$ and reflected in Poisson manner at $a$, i.e. $(X_t,R_t),t\geq 0$ under $\pRefR_u$ is 
\begin{align*}(X_t-(\max\{X_{T_i}:T_i\leq t\}-a)^+,(\max\{X_{T_i}:T_i\leq t\}-a)^+),t\geq 0 \text{ under }\p_u.
\end{align*}
%In risk theory context $R_t$ can be interpreted as the cumulative amount of dividends paid out up to time $t$.
%Now we can formulate the following result.

\begin{prop}\label{prop:reflected}
For $a\geq u\geq 0$ and $\a,\b,\gamma\geq 0$ it holds that
\begin{align*}
&\eRef_u \left(e^{-\a\tauR_0^-+\b X_{\tauR_0^-}-\gamma R_{\tauR_0^-}};\tauR_0^-<\infty\right)\\&\qquad =\e\left[e^{-\a T^U}\eRefR_{u+U} \left( e^{-\a\tau_0^-+\b X_{\tau_0^-}-\gamma R_{\tau_0^-}};\tau_0^-<\infty\right)\right]\e e^{-\a T^D+\beta D},\nonumber\\
&\eRefR_{u} \left( e^{-\a\tau_0^-+\b X_{\tau_0^-}-\gamma R_{\tau_0^-}};\tau_0^-<\infty\right)\e e^{-\a T^D+\beta D}\\
&\qquad =\e\left[e^{-\a T^D}\eRef_{u+D} \left( e^{-\a\tauR_0^-+\b X_{\tauR_0^-}-\gamma R_{\tauR_0^-}};\tauR_0^-<\infty\right)\right].\nonumber
\end{align*}
\end{prop}
\begin{proof} Again, the proof follows merely by inspection in a similar way as for the previous results. The first relation can be seen from Figure~\ref{fig:exit_ref},
where the left picture depicts continuous reflection at~$a$ and Poisson observation at~0, and the right picture depicts the corresponding (shifted) Poisson reflection at~$a$ and continuous observation at~0.
In the left picture Poisson observations yield the sequence:
$u,(u+U_1)\wedge a+D_1,((u+U_1)\wedge a+D_1+U_2)\wedge a+D_2,\ldots$. In the right picture the infima in between Poisson reflection epochs are given by $\widehat u\wedge a+D_1,(\widehat u\wedge a+D_1+U_1)\wedge a+D_2,\ldots$, where we choose $\widehat u$ to be distributed as $u+U$. 
These sequences can be complemented with the respective times as in the proof of Proposition~\ref{prop:exit1}. Finally, it is easy to see that $R$ enters the transforms without requiring any changes. The second relation follows accordingly.
\end{proof}
\begin{figure}[h!]
\centering
\includegraphics{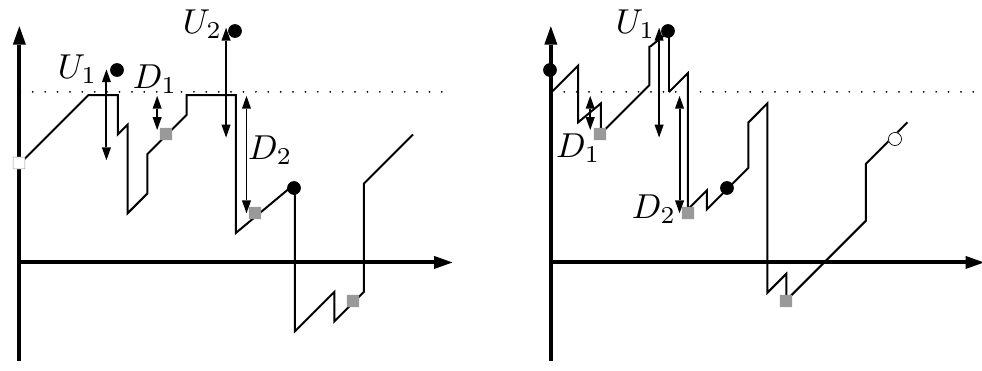}
\caption{Continuous reflection and Poisson exit (left), and Poisson reflection and continuous exit (right)}
\label{fig:exit_ref}
\end{figure}

\begin{remark}\normalfont
It is easy to see that Proposition~\ref{prop:reflected} can be generalized from reflection to so-called refraction. Concretely, consider the processes $X_t-\delta (\overline X_t-a)^+$ and  $X_t-\delta (\max\{X_{T_i}:T_i\leq t\}-a)^+$ for $\delta\in [0,1]$. In the insurance context such a refraction has the interpretation of taxation according to a loss-carry-forward scheme and tax rate $\delta$, see e.g.~\cite{tax_levy}.
\end{remark}
%

%\subsection{Other regimes}
% It was shown in Section~\ref{sec:two_sided} and Section~\ref{sec:refl} that if the observation (reflection) types at each boundary are different, then the corresponding exit problem has a close relation to the one where observation (reflection) types are interchanged. There is no hope to establish similar relationships in other regimes. 
%In this section we discuss the two-sided exit problem with Poisson observation at both barriers and relate it to a non-standard exit problem, where Poisson arrivals are still present.
%The latter problem has an interesting interpretation in risk theory context, and, moreover, it 

\subsection{Two-sided Poisson exit}\label{3.3}
The two-sided exit with Poisson observation at both barriers can be related to a model with another type of exit time. 
%, which in certain variants in the insurance literature is known as a Parisian-type ruin model (see Section~\ref{sec:parisian} for details).
Define the random time $\widetilde \tau_a^+$ of the first observation such that the process has stayed above $a$ during the entire preceding inter-observation period, i.e. $\widetilde \tau_a^+=T_{\widetilde N_a^+}$ with $\widetilde N_a^+:=\min\{i\in\mathbb N:\inf\{X_t,t\in[T_{i-1},T_i]\}>a\}$. Similarly, define $\widetilde \tau_0^-=T_{\widetilde N_0^-}$ with $\widetilde N_0^-:=\min\{i\in\mathbb N:\sup\{X_t,t\in[T_{i-1},T_i]\}<0\}$ as the first observation time such that the process has stayed below $0$ during the entire preceding inter-observation period.
Then for $u\in[0,a]$ it holds that
\begin{equation}\label{eq:same}\p_u(\tauR_0^-<\tauR_a^+)=\e\p_{u+U}(\tau_0^-<\widetilde \tau_a^+)=\e\p_{u+D}(\widetilde\tau_0^-<\tau_a^+).\end{equation}
To see this, one follows the same ideas as above: for the first equality consider infima in between two observations, see Figure~\ref{fig:exit2}, and for the second equality consider suprema in between two observations. 
Similarly, we also have the reverse identities:
\begin{align}
&\p_{u}(\tau_0^-<\widetilde \tau_a^+)=\e\p_{u+D}(\tauR_0^-<\tauR_a^+), \nonumber\\
&\p_{u}(\widetilde\tau_0^-<\tau_a^+)=\e\p_{u+U}(\tauR_0^-<\tauR_a^+).\label{eq:parisian1}
\end{align} 

\subsection{Parisian ruin}\label{sec:parisian}
Parisian ruin is defined as the first time when an excursion of $X$ below 0 is longer than some time $V\geq 0$ (sometimes referred to as {\em implementation delay}). Whereas the classical definition is in terms of a deterministic $V$,   
for analytic tractability it is often assumed that $V$ is a random variable, and that an independent copy of $V$ is assigned to each excursion, see e.g. \cite{loeffen2013parisian} and \cite{lrz}.

Firstly, from the memoryless property it follows that the time $\tauR_0^-$ of Poisson ruin is also the time of Parisian ruin in the case where $V$ is an exponential random variable with rate~$\lambda$. 
Secondly, $\widetilde\tau_0^-$ as defined in Section \ref{3.3} is the time of Parisian ruin in the case where $V$ is Erlang$(2,\lambda)$ distributed (since the latter is the sum of two independent exponential variables).
Similarly to~\eqref{eq:exit1a}, Equation \ref{eq:parisian1} can easily be extended to
\begin{align}\label{eq:parisian_start}&\e_u\left(e^{-\a\widetilde\tau_0^-+\b X_{\widetilde\tau_0^-}};\widetilde\tau_0^-<\infty\right)\\
&\qquad=\e\left[e^{-\a T^U}\e_{u+U}\left(e^{-\a\tauR_0^-+\b X_{\tauR_0^-}};\tauR_0^-<\infty\right)\right]\e e^{-\a T^D+\b D},\nonumber\end{align}
which under the present interpretation relates Parisian ruin quantities with exponential and Erlang(2)-distributed implementation delay (here we took $a=\infty$ for simplicity).
%the later can be further expressed through $\e_x(e^{-\a\tau_0^-+\b X_{\tau_0^-}};\tau_0^-<\infty)$ according to~\eqref{eq:exit1a}.

More generally, consider Parisian ruin with Erlang$(k,\lambda)$ implementation delay and let $\tau_0^{(k)-}$ denote the corresponding ruin time. So in particular $\tau_0^{(0)-}=\tau_0^-$, $\tau_0^{(1)-}=\tauR_0^-$ and $\tau_0^{(2)-}=\widetilde\tau_0^-$. On the other hand, define $\tauR_0^{(k)-}$ as the first epoch $T_i, i\geq k$ such that $X(T_{i-k}),\ldots,X(T_{i})<0$, i.e. the process has been observed negative at the last $k+1$ Poisson epochs.
% We use $\tauR_0^{(k)-}$ to denote the corresponding ruin time; notice that $\tauR_0^{(0)-}=\tauR_0^-$.
Then, along the same line of arguments, we can extend \eqref{eq:parisian_start} (cf. Figure \ref{fig:paris}):
\begin{figure}[h!]
\centering
\includegraphics{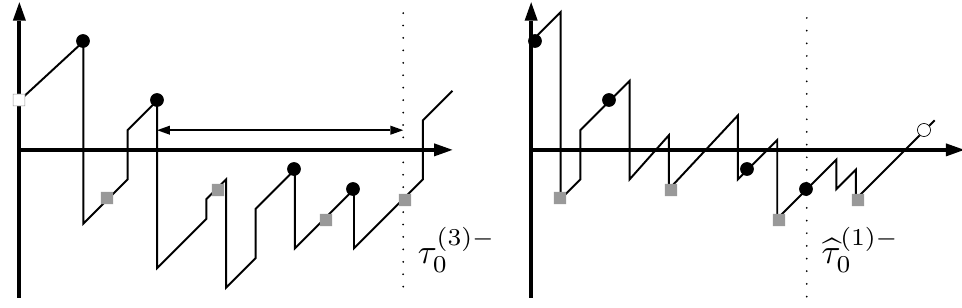}
\caption{Poisson observations are in grey (left picture) and in black (right picture)}
\label{fig:paris}
\end{figure}
%we formulate the following identity relating Parisian ruin in continuous and discrete observation models. 
\begin{prop}
For $k\geq 2$ and $u\geq 0$ we have 
\begin{multline}\label{eq:parisian}\e_u\left(e^{-\a\tau_0^{(k)-}+\b X_{\tau_0^{(k)-}}};\tau_0^{(k)-}<\infty\right)\\
\qquad=\e\left[e^{-\a T^U}\e_{u+U}\left(e^{-\a\tauR_0^{(k-2)-}+\b X_{\tauR_0^{(k-2)-}}};\tauR_0^{(k-2)-}<\infty\right)\right]\e e^{-\a T^D+\b D}.\end{multline}
and 
\begin{multline}\label{eq:parisian2}\e_u\left(e^{-\a\tauR_0^{(k-2)-}+\b X_{\tauR_0^{(k-2)-}}};\tauR_0^{(k-2)-}<\infty\right)\e e^{-\a T^D+\b D}\\
\qquad=\e\left[e^{-\a T^D}\e_{u+D}\left(e^{-\a\tau_0^{(k)-}+\b X_{\tau_0^{(k)-}}};\tau_0^{(k)-}<\infty\right)\right].\end{multline}
%similarly for the reverse identity. %It must be noted, however, that only the case $k=2$ can be reduced to the classical ruin problem.
\end{prop}

% Finally, let us note that in risk theory context $\widetilde\tau_0^-$ can be interpreted as the time of the Parisian-type ruin, i.e.\ the end time of the first inter-observation period during which the surplus stays below 0, see, e.g.,~\cite{loeffen2013parisian} for a standard Parisian ruin.

%but
%\[\p_x(\tau_a^+<\widetilde T_0^-)=\p_x(\tau_a^+<\sigma_0^-)+\int\p_x(\sigma_0^-<\tau_a^+;X(\sigma_0^-)\in\D y)e^{-\Phi y}\p_0(\tau_a^+<\widetilde T_0^-).\]
%In particular, for $x=0$ we have
%\[\p_0(\tau_a^+<\widetilde T_0^-)(1-\e_0(e^{\Phi X(\sigma_0^-)};\sigma_0^-<\tau_a^+))=\p_0(\tau_a^+<\sigma_0^-)\]
%yielding
%\[\p_0(\tau_a^+<\widetilde T_0^-)=\]
%
\section{The case of spectrally one-sided \levy processes}\label{sec:spectral1}
If $X$ is a one-sided \levy process, some of the identities lead to more explicit forms, and this will allow to retrieve a number of results previously obtained in the literature, now with alternative proofs, revealing some more structure of the formulas. 
Without loss of generality assume that $X$ is a spectrally-negative \levy process, i.e.~ it may only have negative jumps and it is not a non-increasing process. Consider its Laplace exponent
\[\psi(\theta)=\log\e e^{\theta X_1},\quad\theta\geq 0,\]
and put $\psi_\a(\theta)=\psi(\theta)-\a$ for $\a\geq 0$.

\subsection{Preliminaries}
Let us first recall some basic functions which play a fundamental role in exit theory, see e.g.~\cite[Ch.\ 8]{kyprianou}.
Let $\Phi_\a$ be the largest (non-negative) zero of $\psi_\a$, and let $W_\a(u),u\geq 0$ be the so-called scale function: a continuous non-negative function determined by its Laplace transform $\int_0^\infty e^{-\theta u}W_\a(u)\D u=1/\psi_\a(\theta),\theta>\Phi_\a$. In addition, we need a second scale function 
\[Z_\a(u,\theta)=e^{\t u}\left(1-\psi_\a(\t)\int_0^u e^{-\t y}W_\a(y)\D y\right),\quad  u\geq 0,\]
which can be rewritten as 
\begin{equation}\label{eq:ZZ}Z_\a(u,\theta)=\psi_\a(\theta)\int_0^\infty e^{-\theta y}W_\a(u+y)\D y\end{equation}
for $\theta>\Phi_\a$, see also~\cite{poisson_exit}.
The two basic one-sided exit identities under continuous observation are 
\begin{align}
\label{eq:exita}&\e_u e^{-\a\tau_a^+}=e^{-\Phi_\a(a-u)},\quad u\leq a\\
\label{eq:exit0}&\e_u \left(e^{-\a\tau_0^-+\b X(\tau_0^-)},\tau_0^-<\infty\right)=Z_\a(u,\b)-W_\a(u)\frac{\psi_\a(\b)}{\b-\Phi_\a},\quad u\geq 0,
\end{align}

and the Wiener-Hopf factors are given by
\begin{align*}
&\e e^{-\a T^U-\b U}=\frac{\Phi_\l}{\Phi_{\l+\a}+\b}, &\e e^{-\a T^D+\b D}=\frac{\l}{\Phi_\l}\frac{\Phi_{\l+\a}-\b}{\l-\psi_\a(\beta)},
\end{align*}
see e.g. ~\cite[Ch.\ 8]{kyprianou}.
In order to apply Formula \eqref{eq:exit1a} of Proposition~\ref{prop:exit1} to~\eqref{eq:exit0}, we first need the following identities:
\begin{lemma}\label{lem:WZ}
For $u\geq 0$ it holds that
\begin{align*}&\e \left(e^{-\a T^U}W_\a(u+U)\right)=\frac{\Phi_{\l}}{\l}Z_\a(u,\Phi_{\l+\a}),\\
&\e \left(e^{-\a T^U}Z_\a(u+U,\b)\right)=\frac{\Phi_\l}{\Phi_{\l+\a}-\b}\left(Z_\a(u,\b)-\frac{\psi_\a(\b)}{\l}Z_\a(u,\Phi_{\l+\a})\right).
\end{align*}
\end{lemma}
\begin{proof}
Firstly, 
\begin{equation}
 \e (e^{-\a T^U};U\in \D y)=\Phi_\l e^{-\Phi_{\l+\a}y}\D y,\label{expp}
\end{equation}
which can be checked by taking transforms and comparing to the Wiener-Hopf factor.
Hence
\[\e \left(e^{-\a T^U}W_\a(u+U)\right)=\int_0^\infty W_\a(u+y)\Phi_\l e^{-\Phi_{\l+\a} y}\D y=\frac{\Phi_\l}{\psi_\a(\Phi_{\l+\a})}Z_\a(u,\Phi_{\l+\a})\] according to~\eqref{eq:ZZ}. For the first identity it is left to note that $\psi_\a(\Phi_{\l+\a})=\l$.

Using~\eqref{eq:ZZ} several times we can write for $\mu>\b>\Phi_\a$:
\begin{align*}
&\int_0^\infty Z_\a(u+x,\b)e^{-\mu x}\D x=\psi_\a(\b)\int_0^\infty e^{-(\mu-\b) x}\int_0^\infty e^{-\b (y+x)}W_\a(u+y+x)\D y\D x\\
&=\int_0^\infty e^{-(\mu-\b) x}\psi_\a(\b)\int_x^\infty e^{-\b y}W_\a(u+y)\D y\D x\\
&=\int_0^\infty e^{-(\mu-\b) x}\left(Z_\a(u,\b)-\psi_\a(\b)\int_0^x e^{-\b y}W_\a(u+y)\D y\right)\D x\\
&=Z_\a(u,\b)\frac{1}{\mu-\b}-\psi_\a(\b)\frac{1}{\mu-\b}\int_0^\infty e^{-\mu y}W_\a(u+y)\D y\\
&=\frac{1}{\mu-\b}\left(Z_\a(u,\b)-\frac{\psi_\a(\b)}{\psi_\a(\mu)}Z_\a(u,\mu)\right).
\end{align*}
Plugging in $\mu=\Phi_{\l+\a}$ and multiplying by $\Phi_\l$ we obtain the second identity. By analytic extension $\b\geq 0$ can be chosen arbitrarily.
\end{proof}

\subsection{One-sided exit}
Assume that $\e X_1=\psi'(0)>0$ and consider the survival probabilities \eqref{surv}.
It is well known that $\phi(u)=\psi'(0)W_0(u)$, see also~\eqref{eq:exit0}. According to Proposition~\ref{prop:survival} and Lemma~\ref{lem:WZ} we have
\[\phiR(u)=\e\phi(u+U)=\psi'(0)\e W_0(u+U)=\psi'(0)\frac{\Phi_\l}{\l}Z_0(u,\Phi_\l),\]
which is Corollary 1 of ~\cite{landriault_occupation}. 

\begin{remark}\normalfont
Note that due to \eqref{expp}, for the spectrally negative \levy process the identity \eqref{rel1} simplifies to the pleasant form
\[\phiR(u)=\e\phi(u+e_{\Phi_\l}),\]
where $e_{\Phi_\l}$ is an exponential random variable with parameter ${\Phi_\l}$. This for instance immediately explains why for a compound Poisson process $X$ with exponential jump sizes the discrete Poisson observation changes the classical ruin probability formula just by a multiplicative factor $\Phi_\l/(\Phi_\l+R_0)$, where $R_0$ is the Lundberg adjustment coefficient (cf. \cite[Eq.2.18]{saj13}).  
\end{remark}

Using the standard identity~\eqref{eq:exit0}, Proposition~\ref{prop:exit1} and Lemma~\ref{lem:WZ} we obtain
\begin{align}
&\e_u \left(e^{-\a\tauR_0^-+\b X(\tauR_0^-)},\tauR_0^-<\infty\right)\label{eq:tauR0}\\&=\e\left(e^{-\a T^U}Z_\a(u+U,\b)-e^{-\a T^U}W_\a(u+U)\frac{\psi_\a(\b)}{\b-\Phi_\a}\right)\e e^{-\a T^D+\b D}\nonumber\\
&=\left(\frac{\Phi_\l}{\Phi_{\l+\a}-\b}\left(Z_\a(u,\b)-\frac{\psi_\a(\b)}{\l}Z_\a(u,\Phi_{\l+\a})\right)-\frac{\Phi_{\l}}{\l}Z_\a(u,\Phi_{\l+\a})\frac{\psi_\a(\b)}{\b-\Phi_\a}\right)\nonumber\\
&\times\frac{\l}{\Phi_\l}\frac{\Phi_{\l+\a}-\b}{\l-\psi_\a(\beta)}=\frac{\l}{\l-\psi_\a(\beta)}\left(Z_\a(u,\b)-Z_\a(u,\Phi_{\l+\a})\frac{\psi_\a(\b)}{\l}\frac{\Phi_{\l+\a}-\Phi_\a}{\b-\Phi_\a}\right)\nonumber.
\end{align}
Also, by considering Proposition~\ref{prop:exit1} for the negative of $X$, see also Corollary~\ref{cor:exit2}, we 
arrive at
\begin{align*}
&\e_u\left(e^{-\a \tauR_a^+-\b (X_{\tauR_a^+}-a)};\tauR_a^+<\infty\right)\\
&=\e\left[e^{-\a T^D}\e_{u+D}\left(e^{-\a \tau_a^+-\b (X_{\tau_a^+}-a)};\tau_a^+<\infty\right)\right]\e e^{-\a T^U-\b U}\\
&=\e e^{-\a T^D-\Phi_\a(a-u-D)}\e e^{-\a T^U-\b U}=e^{-\Phi_\a(a-u)}\frac{\l}{\Phi_\l}\frac{\Phi_{\l+\a}-\Phi_\a}{\l-\psi_\a(\Phi_\a)} \frac{\Phi_\l}{\Phi_{\l+\a}+\b}\\&=e^{-\Phi_\a(a-u)}\frac{\Phi_{\l+\a}-\Phi_\a}{\Phi_{\l+\a}+\b},
\end{align*} 
because of~\eqref{eq:exita} and the fact that $X_{\tau_a^+}=a$ on $\tau_a^+$.
These two identities were obtained in~\cite[Thm.\ 3.1]{poisson_exit} by virtue of a rather technical argument using the expression for the potential density of~$X$.

\subsection{Parisian ruin}
Finally, we relate our results to previous literature on Parisian ruin. Firstly, taking $u=0$ and $\beta=0$ in \eqref{eq:tauR0} (use $Z_\alpha(0,\beta)=1$) one retrieves Corollary 3.2 of \cite{lrz}, which is based on exponential implementation delay. Furthermore, from the form of \cite[Eq.49]{lrz} one can, after some lengthy calculations, obtain the following expression for 
% have 
% established a formula for the transform of the Parisian ruin time for 0 initial capital and mixed Erlang implementation delay (in the case when $X$ has paths of bounded variation on compacts), see~Eq.\ (49) therein. 
Erlang$(2,\lambda)$ implementation delay:
\begin{equation}\label{eq:lrz}\e_0\left(e^{-\a\tau_0^{(2)-}};\tau_0^{(2)-}<\infty\right)=
\frac{\l}{\l+\a}-\frac{\a}{(\l+\a)^2}\frac{\Phi_{\l+\a}-\Phi_\a}{\Phi_\a}\frac{\Phi_{\l+\a}}{\Phi'_{\l+\a}},\end{equation}
where the derivative is with respect to the subindex. 
We can alternatively obtain \eqref{eq:lrz} directly using the results of this paper: \eqref{eq:parisian} and \eqref{eq:tauR0} imply 
\begin{align*}
&\e_0\left(e^{-\a\tau_0^{(2)-}};\tau_0^{(2)-}<\infty\right)=
\e\left[e^{-\a T^U}\e_{U}\left(e^{-\a\tauR_0^-};\tauR_0^-<\infty\right)\right]\e e^{-\a T^D}\\
&=\frac{\l}{\l+\a}\e\left[e^{-\a T^U}\left(Z_\a(U,0)-Z_\a(U,\Phi_{\l+\a})\frac{\a}{\l}\frac{\Phi_{\l+\a}-\Phi_\a}{\Phi_\a}\right)\right]\e e^{-\a T^D}.
\end{align*}
Using Lemma~\ref{lem:WZ} and noting again that $Z_\a(0,\beta)=1$ we get
\begin{align*}
&\e\left(e^{-\a T^U}Z_\a(U,0)\right)=\frac{\Phi_\l}{\Phi_{\l+\a}}\frac{\a+\l}{\l},\\
&\e\left(e^{-\a T^U}Z_\a(U,\Phi_{\l+\a})\right)=\lim_{h\downarrow 0}\frac{\Phi_\l}{\Phi_{\l+\a}-\Phi_{\l+\a-h}}\frac{h}{\l}=\frac{\Phi_\l}{\l\Phi'_{\l+\a}},
\end{align*}
which together with the expression for the Wiener-Hopf factor $\e e^{-\a T^D}$ readily yields~\eqref{eq:lrz}.

\section*{Acknowledgements} The authors would like to thank Ton Dieker for stimulating discussions on the topic. 

%\bibliographystyle{abbrv}
%\bibliography{exit_Poisson_obs}
\end{document}